\documentclass[12pt]{article}
\usepackage{fancyhdr}

\usepackage{amsmath}
\usepackage{mathtools}
\usepackage{amsfonts}
\usepackage{amsthm}
\usepackage{amssymb}
\usepackage{color}

\usepackage{color}

\title{Existence of solutions to degenerate parabolic problems with two weights via the Hardy inequality}

\usepackage{authblk}

\author[$\diamond$]{ Iwona Skrzypczak \footnote{ email address: iskrzypczak@mimuw.edu.pl }}
\author[$\diamond$]{ Anna Zatorska--Goldstein \footnote{ email address: azator@mimuw.edu.pl \\
The research of A.Z.-G. has been supported by the Foundation for Polish Science grant no. POMOST BIS/2012-6/3 }}
\affil[$\diamond$]{\small
Institute of Applied Mathematics and Mechanics,
University of Warsaw, \newline
ul. Banacha 2, 02-097 Warsaw, Poland
}
\date{}
\begin{document}
\maketitle \sloppy

\thispagestyle{empty}

\renewcommand{\it}{\sl}
\renewcommand{\em}{\sl}

\belowdisplayskip=18pt plus 6pt minus 12pt \abovedisplayskip=18pt
plus 6pt minus 12pt
\parskip 4pt plus 1pt
\parindent 0pt

\newcommand{\barint}{
         \rule[.036in]{.12in}{.009in}\kern-.16in
          \displaystyle\int  }
\def\R{{\mathbb{R}}}
\def\r{{\mathbb{R}}}
\def\n{{\mathbb{N}}}
\def\rn{{\mathbb{R}^{N}}}
\def\rN{{\mathbb{R}^{N}}}
\def\rom{{\mathbb{R}^{m}}}
\def\w{\widetilde}
\def\M{{\cal M}}
\def\zi{[0,\infty )}
\def\sf{{\rm supp\,\phi}  }
\def\sx{{\rm supp\,\xi}  }
\def\rp{{\mathbb{R}_{+}}}
\def\bA{{\bar{A}}}
\def\O{{\Omega}}
\def\o{{\omega}}
\def\Oc{{\Omega\cap}}
\def\a{{\alpha}}
\def\b{{\beta}}
\def\e{{\epsilon}}
\def\d{{\delta}}
\def\D{{\Delta}}
\def\dL{{\Delta_p^{\omega_2}}}
\def\t{{\tau}}
\def\s{{\sigma}}
\def\g{{\gamma}}
\def\k{{\kappa}}
\def\tu{{{u}}}
\def\tW{{{W}}}
\def\na{{\nabla}}
\def\rr{{\rangle }}
\def\ll{{\langle }}
\def\irn{{\int_\rn}}
\def\dsp{\def\baselinestretch{1.37}\large
\normalsize}

\newtheorem{theo}{\bf Theorem}[section]
\newtheorem{coro}{\bf Corollary}[section]
\newtheorem{lem}{\bf Lemma}[section]
\newtheorem{rem}{\bf Remark}[section]
\newtheorem{defi}{\bf Definition}[section]
\newtheorem{ex}{\bf Example}[section]
\newtheorem{fact}{\bf Fact}[section]
\newtheorem{prop}{\bf Proposition}[section]

\newcommand{\ds}{\displaystyle}
\newcommand{\ts}{\textstyle}
\newcommand{\ol}{\overline}
\newcommand{\wt}{\widetilde}
\newcommand{\ck}{{\cal K}}
\newcommand{\ve}{\varepsilon}
\newcommand{\vp}{\varphi}
\newcommand{\pa}{\partial}
\newcommand{\psa}{\Phi_{s,A}}


\parindent 1em

\begin{abstract}
The paper concentrates on the application of the following Hardy inequality
\begin{equation*}
\int_\O \ |\xi(x)|^p \omega_{1 }(x)dx\le \int_\O |\nabla \xi(x)|^p\omega_{2 }(x)dx,
\end{equation*}
to the proof of  existence of weak solutions to degenerate parabolic problems of the type
\begin{equation*} \left\{\begin{array}{ll}
 u_t-div(\omega_2(x)|\nabla u|^{p-2} \nabla u )=  \lambda  W(x) |u|^{p-2}u&  x\in\Omega,\\
 u(x,0)=f(x)& x\in\Omega,\\
 u(x,t)=0& x\in\partial\Omega,\ t>0,\\
\end{array}\right.
\end{equation*}
on an open subset $\Omega\subseteq\mathbb{R}^n$, not necessarily bounded, where \[W(x)\leq \min\{m,\omega_1(x)\},\qquad m\in\R_+.\]
\end{abstract}

\smallskip



\section{Introduction}

We investigate existence and regularity of solutions to a  broad class of nonlinear parabolic equations
\begin{equation} \label{paraprob0}
u_t-div(\omega_2(x)|\nabla u|^{p-2}\nabla u)=\lambda \omega_1(x) |u|^{p-2}u,\qquad \text{in}\quad\Omega
\end{equation}
on an open subset $\Omega\subseteq\mathbb{R}^n$, not necessarily bounded, with certain weights $\omega_1,\omega_2\ge 0$. We explore the meaning of the optimal constant in~the Hardy inequality in~parabolic problems, see~\cite{gaap} for details. The optimal constant in~classical versions of the Hardy inequality indicates the  critical $\lambda$ for~blow--up or global existence, as well as the sharp decay rate of~the solution, see e.g.~\cite{AdChaRa,anh,bargold,BV,bh,32,gaap,xiang,vz}, Section 2.5 in \cite{ma}). In the elliptic case, existence results can also be obtained via the Hardy inequality, e.g.~\cite{drabekgarciahuidobro,pucci,puccimanahuidmana}.

The inequalities which are crucial in our approach used to be called improved Hardy inequalities, improved Hardy--Sobolev inequalities, or Hardy--Poincar\'{e} inequalities. One of the first well--known `improvents'  was introduced in~\cite{BV} as
 \[C_1\int_\Omega \frac{u^2}{|x|^2} dx+C_2\left(\int_\Omega u^q dx\right)^\frac{2}{q}\leq\int_\Omega |\nabla u|^2 dx\]
and was used therein in investigations on qualitative properties to~the equation $-\Delta u = \lambda f(u)$, with convex and increasing function $f$.

In~\cite{vz} Vazquez and Zuazua describe the asymptotic behaviour of the heat equation that reads \[u_t=\D u +V(x)u \quad\mathrm{and}\quad \D u +V(x)u+\mu u=0,\]
where $V(x)$ is an inverse--square potential (e.g. $V(x)=\frac{\lambda}{|x|^2}$).
The authors consider solutions to the Cauchy--Dirichlet problem in a bounded domain and to the Cauchy problem in $\rn$ as well. The key tool is an improved form of~the Hardy--Poincar\'{e} inequality.
Furthermore, in~\cite{vz} the authors  generalize the seminal paper by Baras and Goldstein~\cite{bargold}. Nevertheless, the involved Hardy inequalities do  not admit a broad class of weights.

 In several papers, e.g.~\cite{bbdg,blanchet_09,sharp}, dealing with the rate of convergence of~solutions to fast diffusion equations \[u_t=\Delta u^m,\] the authors study the estimates for the constants in Hardy-Poincar\'{e}-type inequalities.   The optimal constant in Hardy-type inequalities used to indicate the  critical $\lambda$ for blow-up or global existence, as well as the sharp decay rate of the solution. However, they usually deal with weights of the form
 \[|x|^{-\alpha}\quad\mathrm{or}\quad\left(1+|x|^2\right)^{-\alpha}.\]

The inspiration of our research was the paper of  Garc\'{i}a Azorero and  Peral Alonso~\cite{gaap}, who apply the Hardy inequality  {\cite[Lemma 2.1]{gaap}} of the form
\[\lambda_{N,p}\int_{\rN} \ |\xi|^p |x|^{-p}  \ dx \le  \int_{\rN} |\nabla \xi|^p\, dx ,\]
where $\lambda_{N,p}$ is optimal, to analyse positivity of the following nonlinear operator
\begin{equation*}\label{operatorgaap} {\cal L}_\lambda u = -\Delta_p u-\frac{\lambda}{|x|^p}|u|^{p-2}u\end{equation*}
in $W^{1,p}_0(\Omega)$ and to obtain the existence of weak solutions to the corresponding parabolic problem.

Our proof follows classical methods of Lions~\cite{lions}, as well as the approach of~Anh,~Ke~\cite{anh}, who consider the initial boundary value problem for a class of~quasilinear parabolic equations involving weighted $p$-Laplace operator. Our major difficulties are of technical nature and require more advanced setting, i.e. two-weighted Sobolev spaces $W^{1,p}_{(\omega_1,\omega_2)}(\Omega)$, due to presence of~general class of weights both in the leading part of~the operator as well as on the~right-hand side of~\eqref{paraprob0}.

To involve broader class of weights in~\eqref{paraprob0}, we need more general Hardy inequalities. We apply the ones derived in~\cite{plap}, having the form \begin{equation*}
\int_\O \ |\xi(x)|^p \omega_{1 }(x)dx\le \int_\O |\nabla \xi(x)|^p\omega_{2 }(x)dx,
\end{equation*}
where the involved weights $\omega_{1 },\omega_2$ depend on a weak solution to PDI:
\begin{equation*}
-\Delta_pv\ge \Phi \quad \mathrm{in}\quad \O,
\end{equation*}
with a locally integrable function $\Phi$ (see Theorem \ref{theoplap}). Quite a general function
$\Phi$ is allowed. It can be negative or sign changing if only it is, in~a~certain sense, bounded from below. The weights $\omega_1,\omega_2$ in~\eqref{paraprob0} are assumed to~be a pair in the Hardy inequality. This inequality is needed to obtain a~priori estimates for solutions to~\eqref{paraprob0}.

Moreover, to ensure that the weighted Sobolev spaces $W^{1,p}_{(\omega_1,\omega_2)}(\Omega)$ and $L^p(0,T;W^{1,p}_{(\omega_1,\omega_2)}(\O))$, with two different weights $\omega_1,\omega_2$, have proper structure, we impose additional regularity restrictions on the weights, see Subsection~\ref{ss:weights}. We suppose $2\leq p< N$ to deal with embeddings.

In the paper we add yet another restriction, namely we consider the problem
 \begin{equation*} \label{paraprob} \left\{\begin{array}{ll}
 u_t-div(\omega_2(x)|\nabla u|^{p-2} \nabla u )=  \lambda  W(x) |u|^{p-2}u&  x\in\Omega,\\
 u(x,0)=f(x)& x\in\Omega,\\
 u(x,t)=0& x\in\partial\Omega,\ t>0,\\
\end{array}\right.
\end{equation*}
on an open subset $\Omega\subseteq\mathbb{R}^n$, not necessarily bounded, where the function $W(x)\leq \min\{m,\omega_1(x)\}$, $m\in\R_{+}$. The case of unbounded potential on the right-hand side is more difficult and requires more complex arguments. This work is in progress. We decided to add the restriction here in order to~make the presentation more transparent.

\section{Preliminaries}\label{prelim}

\subsection{Notation}

In the sequel we assume that  $2\leq p<N$, $\frac{1}{p}+\frac{1}{p'}=1$, $\O\subseteq\rn$ is an open subset, not necessarily bounded. For $T>0$ we denote $\Omega_T=\Omega\times(0,T)$.

We denote the $p$-Laplace operator by
\[
 \Delta_p  u =  \mathrm{div}( |\nabla u|^{p-2}\nabla u)\] and the $\omega$-$p$-Laplacian by \begin{eqnarray}
 \Delta_p^{\omega } u&=& \mathrm{div}(\omega|\nabla u|^{p-2}\nabla u),\label{Lpom2}
\end{eqnarray}
 with a certain weight function $\omega:\Omega\to\R$.

By $\langle f,g\rangle$ we denote the standard scalar product in $L^2(\Omega)$.

\subsection{Sobolev spaces}

Suppose $\omega$ is a positive, Borel  measurable, real  function  defined on an open set $\Omega\subseteq \rn$,  satisfying the so-called $B_p$ condition, i.e.
 \begin{equation}
 \label{Bp}
 \omega^{-1/(p-1)}\in L^1_{{  loc}}(\Omega),
 \end{equation}
 see~\cite{kuf-opic}. This condition is weaker than the $A_p$-condition, see~\cite{muckenhoupt}.

 Whenever $\omega_1,\omega_2$ satisfy~\eqref{Bp}, we denote
\begin{equation}\label{polnorma}
W^{1,p}_{(\omega_1,\omega_2)}(\Omega):= \left\{ f\in L^{p}_{\omega_1}(\Omega) :  
\nabla f\in (L_{\omega_2}^p (\Omega ))^N\right\},
\end{equation}
where $\nabla$ stands for the distributional gradient. The space is equipped with the norm
\begin{multline*}
\| f\|_{W^{1,p}_{(\omega_1,\omega_2)}(\O)}\ :\,=\ \| f\|_{L^{p}_{\omega_1}(\O)} + \| \nabla f\|_{(L_{\omega_2}^p (\Omega ))^N}\\
=\left(\int_{\Omega}|f|^{p}{\omega_1(x)}dxdt\right)^\frac{1}{p} + \left(\int_{\Omega}\sum_{i=1}^N\left|\frac{\partial f}{\partial x_i} \right|^{p}{\omega_2(x)}dxdt\right)^\frac{1}{p}.
\end{multline*}

 \begin{fact}[e.g.~\cite{kuf-opic}]\label{factemb} \rm If  $p>1$, $\Omega\subset \rn$ is an open set, $\omega_1,\omega_2$ satisfy~\eqref{Bp}, then \begin{itemize}
 \item $W^{1,p}_{(\omega_1,\omega_2)}(\Omega)$ defined by \eqref{polnorma} equipped with the norm $\| \cdot\|_{W_{(\omega_1,\omega_2)}^{1,p}(\Omega)}$ is a~Banach space;
 \item $\displaystyle L^{p}_{\omega_1,loc}(\Omega)\subseteq  L^1_{{ loc}}(\Omega);$
  \item  $\displaystyle \overline{Lip_0(\Omega)}  =\overline{C^\infty_0(\Omega)}  =W^{1,p}_{(\omega_1,\omega_2),0}(\O),$ where the closure is in the norm $\| \cdot\|_{W_{(\omega_1,\omega_2)}^{1,p}(\Omega)}$;
\item if $\omega_1,\omega_2$ are a pair in the Hardy-Poincar\'{e} inequality of the form~\eqref{inq:hardyplap}, we may consider the Sobolev space $W^{1,p}_{(\omega_1,\omega_2),0}(\O)$ equipped with the norm
\[\| f\|_{W^{1,p}_{(\omega_1,\omega_2), 0}(\O)} = \| \nabla f\|_{L^{p}_{\omega_2}(\O)}.\]
  \end{itemize}
\end{fact}

\begin{fact}\label{facthemi} \rm
The operator $\dL$, given by~\eqref{Lpom2}, is hemicontinuous, i.e. for all $u,v,w\in W^{1,p}_{(\omega_1,\omega_2),0}(\O)$ the mapping $\lambda\mapsto\ll \dL(u+\lambda v),w\rr$ is continuous from $\r$ to $\r$.
\end{fact}

\bigskip

We look for solutions in the space $L^p(0,T;W^{1,p}_{(\omega_1,\omega_2)}(\O))$, i.e.
\begin{equation*}
\label{LpWp}L^p(0,T;W^{1,p}_{(\omega_1,\omega_2)}(\O))=\left\{f\in L^p(0,T;L^{p}_{\omega_1}(\O)):\nabla f\in (L^p(0,T;L^{p}_{\omega_2}(\O)))^N\right\}
\end{equation*}
(as before $\nabla$ stands for the distributional gradient with respect to the spacial variables), equipped with the norm
\begin{equation*}
\label{normLpWp}
\| f\|_{L^p(0,T;W^{1,p}_{(\omega_1,\omega_2)}(\O))}:= \left(\int_{0}^T\| f\|_{L^{p}_{\omega_1}(\O)}dt\right)^\frac{1}{p} + \left(\int_{0}^{T}\| \nabla f\|_{(L_{\omega_2}^p (\Omega ))^N}dt\right)^\frac{1}{p}.
\end{equation*}

\subsubsection*{Dual space}

By $W^{-1,p'}_{(\omega_1,\omega_2)}(\O)$ we denote the dual space to $W^{1,p}_{(\omega_1,\omega_2),0}(\O)$ and the duality pairing is given by the standard scalar product.

We note that $ L^{p'}(0,T;W^{-1,p'}_{(\omega_1,\omega_2)}(\O))$  is the dual space to $L^p(0,T;W^{1,p}_{(\omega_1,\omega_2),0}(\O))$.

\subsubsection*{Embeddings} In the framework of weighted Sobolev spaces we impose restrictions on the weights to ensure that we have the Sobolev--type embeddings. Following~\cite{anh}, we introduce the following conditions:
\begin{itemize}
\item[$({\cal H}_\alpha)$]  There exists  some $\alpha\in (0,p)$
\begin{equation}
\label{Halpha} \liminf_{x\to z}\frac{\omega }{|x-z|^{\alpha}}>0\qquad \forall_{z\in\overline{\Omega}}.
\end{equation}
\item[$({\cal H}^\infty_{\alpha,\beta})$] There exists some $\beta> p+\frac{N}{2}(p-2)$\begin{equation}
\label{Hinfty} \liminf_{x\to z}\frac{\omega }{|x|^{-\beta}}>0\qquad \forall_{z\in\overline{\Omega}}.
\end{equation}
\end{itemize}

We have the following result, which is a direct consequence of~\cite[Proposition~2.1]{anh}.

\begin{prop} \label{prop:emb} Suppose $2\leq p<\infty$, the function $\omega_2:\Omega\to\R$ is locally intergrable and satisfies condition $({\cal H}_\alpha)$ and, if $\Omega$ is unbounded, we additionally require that $\omega_2$ satisfies also condition $({\cal H}^\infty_{\alpha,\beta})$.
 Assume further that  $(\omega_1,\omega_2)$ is a  pair of weights in the Hardy inequality~\eqref{inq:hardyplap}.

Then for each $r\in \left[1,\frac{pN}{N-p+\alpha}\right)$ we have\begin{equation}
\label{eq:Lremb}
W^{1,p}_{(\omega_1,\omega_2),0}(\O)\subset\subset L^r(\Omega).
\end{equation}
If, additionally, for arbitrary $U\subset\subset \Omega$ there exists a constant $c_U$ such that $\omega_2(x)\geq c_U>0$ in $U$, then we can choose $\alpha=0$.

In particular,  under the above conditions, we have\begin{equation}
\label{eq:L2emb}
W^{1,p}_{(\omega_1,\omega_2),0}(\O)\subset\subset L^2(\Omega)
\end{equation}
and \begin{equation*}
\label{eq:L2L2emb}
L^p(0,T;W^{1,p}_{(\omega_1,\omega_2),0}(\O))\subset  L^2(0,T;L^2(\Omega))= L^2(\Omega_T).
\end{equation*} \end{prop}


\subsection{The weights}\label{ss:weights}

In this section we give the restrictions on $\omega_1$ and $\omega_2$ sufficient for the existence of solutions to the problem
\begin{eqnarray}
\label{operatorgeneral}
u_t-\mathrm{div}(\omega_2|\na u|^{p-2}\na u) &=& \lambda\omega_1|u|^{p-2}u\qquad on\quad\Omega_T.
\end{eqnarray}

We call the pair of functions $(\omega_1,\omega_2)$  an admissible pair in our framework if the following conditions are satisfied
\begin{enumerate}
\item $\omega_1,\omega_2:\overline{\Omega}\to\R_+\cup\{0\}$ and  $ \omega_2$ is such that for any  $U\subset\subset \Omega$ there exists a constant $\omega_2(x)\geq c_U>0$ in $U$;
\item $\omega_1,\omega_2$ satisfy the $B_p$-condition~\eqref{Bp};
\item $\omega_2$ satisfies~$({\cal H}_\alpha)$ and if $\Omega$ is unbounded we additionally require that $\omega_2$ satisfies also condition $({\cal H}^\infty_{\alpha,\beta})$;
\item $(\omega_1,\omega_2)$ is a  pair of weights in the Hardy inequality~\eqref{inq:hardyplap}.

\end{enumerate}

\subsubsection*{Comments}

We give here the reasons for which we assume the above conditions, respectively:
\begin{enumerate}
\item is necessary for the strict monotonicity of the operator;
\item is necessary for   $W^{1,p}_{(\omega_1,\omega_2)}(\O)$ to be a Banach space;\\ it implies $\omega_1,\omega_2\in L^{1}_{loc}(\Omega)$;
\item is necessary for the compact embedding~\eqref{eq:Lremb};\\
 in particular it provides the existence of the basis of $W^{1,p}_{(\omega_1,\omega_2),0}(\O)$, which is orthogonal in $L^2(\Omega)$;
\item is necessary for a priori estimates for solutions.
\end{enumerate}

\subsubsection*{Examples} We give several examples of  weights admissible in our setting. Following~\cite{plap,bcp-plap}, we can take
\begin{itemize} 
\item on $\rn\setminus\{0\}$\begin{itemize}
\item $\o_1(x)=|x|^{\gamma-p}$, $\o_2(x)=|x|^{\gamma},$ for $\gamma< p-N$, with the optimal constant $\lambda_{N,p}=\left( ({p-N-\gamma})/p\right)^p;$\end{itemize}
\item on $\rn $\begin{itemize}
\item $\o_1(x)=\left(1+| x|^{\frac{p}{p-1}}\right)^{(p-1)(\gamma-1)}$, $\o_2(x)=\left(1+| x|^{\frac{p}{p-1}}\right)^{(p-1) \gamma },$ for~$\gamma>1$, with   $ {K=n\left(\frac{p(\gamma-1)}{p-1}\right)^{p-1}}$ optimal whenever $\gamma\geq n+ 1-\frac{n}{p};$
\end{itemize}
\item on $\Omega\subseteq\rn$

Let us consider any function $u$  that is superharmonic in $\Omega\subseteq\rn$ (i.e.~$\Delta u\leq 0$) and an arbitrary $\beta>3$. We can take  \[\o_1(x)=u^{-\beta-1}(x)|\nabla u(x)|^2\quad {\rm and}\quad \o_2(x)=u^{-\beta+1}(x),\] if only each of them satisfies $B_p$ condition~\eqref{Bp}. Then~$K={3(\beta-3)}.$

\end{itemize}

\subsubsection*{General Hardy inequality}

The following result {\cite[Theorem 4.1]{plap}} gives sufficient conditions for the Hardy-type inequality to hold.

\begin{theo}
\label{theoplap} Let $\O$ be any open subset of $\rn$,  $1<p <\infty$, and nonnegative function $v
\in W^{1,p}_{loc}(\Omega)$ such that  $-\Delta_p v\in L^1_{loc}(\Omega)$. Suppose that the following condition is satisfied \begin{equation*}\label{s0}
 \sigma_0:=\inf \left\{\sigma\in\R: {-\Delta_p v\cdot v}+{\sigma |\nabla v|^p}\geqslant 0 \quad a.e.\ in\ \Omega\cap \{v>0\}\ \right\}\in\R.
\end{equation*} Moreover, let $\beta$ and $\sigma$ be arbitrary real numbers such that $\beta>\min\{0,\sigma\}$.

Then, for every Lipschitz function $\xi$ with compact support in $\Omega$, we have
\begin{equation}\label{inq:hardyplap}
K \int_\Omega \ |\xi|^p \omega_{1}(x) dx \leqslant \int_\Omega |\nabla \xi|^p \omega_{2}(x)dx,
\end{equation} where $K=\left(\frac{\beta-\sigma}{p-1}\right)^{p-1}$,\begin{eqnarray}
&\omega_{1}(x)&= \big(-\Delta_p v\cdot v+\sigma |\nabla v|^p\big)\cdot v^{-\beta-1}\chi_{\{v>0\}},\label{om1p}
\\\label{om2p}
&\omega_{2}(x)&=  v^{p-\beta-1}\chi_{\{|\nabla v|\neq 0\}}.
\end{eqnarray}
\end{theo}

\begin{rem}\rm  Note that under the assumptions of Theorem~\ref{theoplap}  we have $\omega_1,\omega_2\in L^1_{loc}(\Omega)$.
 \end{rem}

\begin{rem}\rm \label{rem:const}  By $\lambda_{N,p}$ we denote the greatest possible constant $K$ such that~\eqref{inq:hardyplap} holds (for fixed weights $\omega_1,\omega_2$).

Let us mention several examples of application of~\cite[Theorem 4.1]{plap} leading to the inequalities with the best constants. Namely, they are achieved in
 the classical Hardy inequality (Section~5.1 in~\cite{plap});
 the Hardy-Poincar\'{e} inequality obtained  in~\cite{bcp-plap} due to~\cite{plap}, confirming some constants from~\cite{gm} and~\cite{bbdg} and establishing the optimal constants in further cases;
 the Poincar\'{e} inequality concluded from~\cite{plap}, confirmed to hold with best constant in Remark~7.6 in \cite{akiraj}.
Moreover, the  inequality in Theorem~5.5 in~\cite{plap} can also be retrieved by the methods of~\cite{akkppstudia} with the same constant, while some inequalities from Proposition~5.2  in~\cite{akkppcentue} are comparable with Theorem~5.8 in~\cite{plap}. Generalisation of~\cite[Theorem 4.1]{plap} in~\cite{pdakis1} leads to the optimal result in~\cite{akisP}.\end{rem}

\section{Existence}

Let us start with the definition of a weak solution to the parabolic problem.

We consider \[W:\Omega\to\R_{+}\]
such that
\[W(x)\leq \min\{m,\omega_1(x)\}\]
 with a certain $m\in\R_+$.

\begin{defi}
We call a function $u$ a weak solution to
\begin{equation}\label{eq:main}\left\{\begin{array}{ll}
 u_t-\dL u=  \lambda W(x)|u|^{p-2}u & x\in\Omega,\\
 u(x,0)=f(x)& x\in\Omega,\\
 u(x,t)=0& x\in\partial\Omega,\ t>0,\\
\end{array}\right.
\end{equation} if
\begin{align*}
u &\in L^p(0,T; W_{(\omega_1,\omega_2),0}^{1,p}(\Omega)),\\
u_t &\in L^{p'}(0,T; {W^{-1,p'}_{(\omega_1,\omega_2)}(\Omega)}),
\end{align*}
and
\begin{equation*}
\label{eq:mainweak}
\int_{\Omega_T}\left(  u_t\xi+\omega_2|\nabla u|^{p-2} \nabla u \nabla \xi +\lambda W(x) |u|^{p-2}u \xi\right)dx\,dt=0,
\end{equation*}
holds for every $\xi\in L^p(0,T; W_{(\omega_1,\omega_2),0}^{1,p}(\Omega))$.
\end{defi}

Existence of a solution to the  problem is obtained by the Galerkin approximation, where we use the fact that the operator $-\dL$ from~\eqref{Lpom2} is monotone (note that it is implied by $\omega_2>0$ in $\Omega$).

\begin{theo}
 \label{lemextrun} Suppose $2\leq p<N$, $\Omega\subseteq\rn$ is an open subset, $f\in L^2(\Omega)$, $m\in\R_+$. Assume that $\omega_1,\omega_2:\Omega\to\R_+$ satysfying~\eqref{Bp} are given by~\eqref{om1p},\eqref{om2p}, respectively. Moreover, assume $\omega_2$ satisfies~$({\cal H}_\alpha)$ and if $\Omega$ is unbounded additionally require that $\omega_2$ satisfies also condition $({\cal H}^\infty_{\alpha,\beta})$.

 There exist  $\lambda_0=\lambda_0(p,N,\omega_1,\omega_2)$ and a weak solution $u$ to~\eqref{eq:main}, such that for all $\lambda\in(0,\lambda_0)$ the solution $u$ is in $L^\infty(0,T; L^2(\Omega_T))$.
\end{theo}
\begin{proof}
We apply the Galerkin method. Remind that according to~\eqref{eq:L2emb} we have \[W^{1,p}_{(\omega_1,\omega_2),0}(\O)\subset  L^2(\Omega),\] which is a and the former is a closed subspace of the latter. Thus, each basis of~$W^{1,p}_{(\omega_1,\omega_2),0}(\O)$ is contained in~the $ L^2(\Omega)$ and  can be orthogonalised with respect to $L^2(\Omega)$ scalar product. Suppose $(e_j)_{j=1}^\infty$ is a basis of~$W_{(\omega_1,\omega_2),0}^{1,p}(\Omega )\cap  L^2(\Omega)$, which is orthogonal in~$L^2(\Omega)$.  We construct an~approximating sequence $(\tu^{n})_{n=1}^{\infty}$ given by
\begin{equation}
\label{eq:tuapprox}
\tu^n(t)=\sum_{k=1}^{n} a_n^k(t)e_k,\qquad t\in[0,T].
\end{equation}
 We determine $a_n^k$ by solving
the differential equation, which is a projection of the original problem to the finite-dimensional subspace $\mathrm{span}(e_1,\dots,e_n)$,\begin{multline}
(a^k_n)'(t)=\label{eq:a'}\\
=  \sum_{j=1}^n\left[\int_\Omega\omega_2 \left|\sum_{l=1}^{n} a^l_n (t)\nabla e_l\right|^{p-2}a^j_n \nabla e_j\nabla e_k dx +\lambda\int_\Omega\tW \left|\sum_{l=1}^{n} a^l_n (t)  e_l\right|^{p-2}a^j_n   e_j  e_k dx \right]
\end{multline}
 with the initial conditions
\[\langle a^k_n(0),e_k\rangle=\langle f,e_k\rangle,\qquad k=1,\dots,n,\]
\[\tu^n (0)=f_n=\sum_{k=1}^n\langle  f, e_k\rangle e_k\in\mathrm{span}(e_1,\dots,e_n).\]

Using the Peano theorem, we get the local existence of the coefficients $a^k_n(t)$ on some intervals $[0,t_n]$. Note that $a^k_n(t)$ depend on $||W||_{L^\infty(\Omega)}\leq m$.

 Let us now establish an a~priori estimate for $\tu^n$ and $(\tu^n)_t$. Because of~\eqref{eq:a'}, we have
 \[\frac{1}{2}\frac{d}{dt}\|\tu^n(t)\|_{L^2(\Omega)}^2+\int_\Omega \omega_2 |\nabla \tu^n|^pdx =\lambda \int_\Omega \tW  |\tu^n|^pdx\]
and moreover
\begin{equation*}
\int_0^T \frac{1}{2}\frac{d}{dt}\|\tu^n(t)\|_{L^2(\Omega)}^2dt=  \frac{1}{2} \|\tu^n(T)\|_{L^2(\Omega)}^2- \frac{1}{2}\|\tu^n(0)\|_{L^2(\Omega)}^2.
\end{equation*}

Therefore,
\[\frac{1}{2}\|\tu^n (T)\|_{L^2(\Omega)}^2+ \int_0^T\|\nabla \tu^n(t)\|^p_{L^p_{\omega_2}(\Omega)}dt  \leq\lambda \int_0^T\|  \tu^n(t)\|^p_{L^p_{\omega_1}(\Omega)}dt+\frac{1}{2}\|\tu^n (0)\|_{L^2(\Omega)}^2. \]
We estimate the right-hand side using the Hardy-type inequality~\eqref{inq:hardyplap} and we obtain
\[\frac{1}{2}\|\tu^n (T)\|_{L^2(\Omega)}^2+ \int_0^T\|\nabla \tu^n(t)\|^p_{L^p_{\omega_2}(\Omega)}dt  \leq \frac{\lambda}{K} \int_0^T\|  \nabla\tu^n(t)\|^p_{L^p_{\omega_2}(\Omega)}dt+\frac{1}{2}\|\tu^n (0)\|_{L^2(\Omega)}^2, \]
which we rearrange to \begin{equation}
\label{eq:apriori}\frac{1}{2}\|\tu^n (T)\|_{L^2(\Omega)}^2+ \left(1-\frac{\lambda}{K}\right)\int_0^T\|\nabla \tu^n(t)\|^p_{L^p_{\omega_2}(\Omega)}dt \leq  \frac{1}{2}\|\tu^n (0)\|_{L^2(\Omega)}^2.
\end{equation}
Thus we can assume $t_n=T$ for each $n$ and as $\tu^n(0)\to f\in L^2(\Omega)$. We conclude that the sequence $(\tu^n)_{n=1}^\infty$ is bounded in \[
 L^\infty(0,T;  L^2(\Omega ))\cap L^p(0,T; {W^{1,p}_{(\omega_1,\omega_2),0}(\Omega)}).\]

Thus, we can choose its subsequence (denoted $(\tu^n)_{n=1}^\infty$ as well) such that
\begin{eqnarray}
\tu^n\xrightharpoonup[n\to\infty]{\ \ *\ \ }\tu& \mathrm{in}&   L^\infty(0,T;  L^2(\Omega )),\nonumber\\
\tu^n\xrightharpoonup[n\to\infty]{\ \ \ \ \ }\tu& \mathrm{in}&   L^p( 0,T; {W^{1,p}_{(\omega_1,\omega_2),0}(\Omega)}),\label{eq:weakconv}\\
\tu^n(T)\xrightharpoonup[n\to\infty]{\ \ \ \ \ }\zeta& \mathrm{in}&   L^2( \Omega).\nonumber
\end{eqnarray}

The fact that $\zeta=u(T)$ is a direct consequence of the arguments of~\cite[Chap.~2, Par.~1.2.2]{lions}.

Let us show that the limit function $\tu$ satisfies~\eqref{eq:main}. As $\tu^n$ solves the finite-dimensional projection  of the problem~\eqref{eq:main},   for each test function  \mbox{$w\in L^p(0,T;{W^{1,p}_{(\omega_1,\omega_2),0}(\Omega)})$} we have
\[\left|\int_0^T\int_\Omega  -\Delta_p^{\omega_2}\tu^n w\,   dxdt\right|=
\left|\int_0^T\int_\Omega   |\nabla\tu^n |^{p-2} \nabla\tu^n \nabla w\, \omega_2\, dxdt\right|=\]
\[=\left|\int_0^T\int_\Omega \left(\omega_2^{\frac{p-1}{p}}|\nabla\tu^n |^{p-2} \nabla\tu^n\right)\left(\omega_2^{\frac{ 1}{p}}\nabla w\right)dxdt\right|\leq\]
\[\leq\left|\int_0^T\left(\int_\Omega  \omega_2|\nabla\tu^n |^{p } dx\right)^{\frac{p-1}{p}}\left(\int_\Omega \omega_2 |\nabla w|^{p } dx\right)^{\frac{1}{p}}dt\right|
\leq
\]
\[
\leq\left(\int_0^T\|\nabla\tu^n \|_{L^p(\Omega,\omega_2)}^{p } dt\right)^{\frac{p-1}{p}}\left( \int_0^T\|\nabla w\|_{L^p_{\omega_2}(\Omega)}^{p } dt\right)^{\frac{1}{p}} =\]\[=\| \tu^n \|_{L^p(0,T;W^{1,p}_{(\omega_1,\omega_2),0}(\Omega))}^{\frac{p-1}{p} }\| w \|_{L^p(0,T;W^{1,p}_{(\omega_1,\omega_2),0}(\Omega))}.\]

Using boundedness of $(\tu^n)_{n=1}^\infty$ in $L^p( 0,T; {W^{1,p}_{(\omega_1,\omega_2),0}(\Omega)})$ we infer that  $(-\Delta_p^{\omega_2}\tu^n)_{n=1}^\infty$
is bounded in $L^{p'}(0,T; {W^{-1,p'}_{(\omega_1,\omega_2)}(\Omega)})$. Moreover, we observe that there exists $\chi\in L^{p'}(0,T; {W^{-1,p'}_{(\omega_1,\omega_2)}(\Omega)})$, such that (up to a subsequence) we have
\begin{eqnarray*}-\Delta_p^{\omega_2}\tu^n\xrightharpoonup[n\to\infty]{\ \ \ \ \ }\chi& \mathrm{in}&  L^{p'}(0,T; {W^{-1,p'}_{(\omega_1,\omega_2)}(\Omega)}).
\end{eqnarray*}

As for $(\tu^n)_t(t)$, for each $w\in L^p(0,T;{W^{1,p}_{(\omega_1,\omega_2),0}(\Omega)})$ we have
\[\int_0^T \int_\Omega \ (\tu^n)_t(t)w \, dx dt = \int_0^T \int_\Omega \left(-\dL \tu^n(t)+ \lambda \tW |\tu^n(t)|^{p-2}\tu^n(t)\right) w \, dx dt .\]
Therefore \begin{eqnarray*} (\tu^n)_t(t)\xrightharpoonup[n\to\infty]{\ \ \ \ \ }\tu_t(t)& \mathrm{in}&  L^{p'}(0,T; {W^{-1,p'}_{(\omega_1,\omega_2)}(\Omega)}).\\
\end{eqnarray*}

Our aim is now to show that $\chi=-\dL\tu$, which finishes the proof. We observe that $-\Delta_p^{\omega_2}$ is a  monotone operator,  therefore for each \mbox{$w\in L^p(0,T;{W^{1,p}_{(\omega_1,\omega_2),0}(\Omega)})$} there holds
\[A^n:=\int^T_0\ll-\dL \tu^n(t)+\dL w(t),\tu^n(t)-w(t)\rr dt\geq 0.\]
Since
\begin{multline*}\int^T_0\ll-\dL \tu^n(t) ,\tu^n(t) \rr dt=\\=\lambda \int_0^T \langle \tW|\tu^n|^{p-2}\tu^n,\tu^n\rangle dt+\frac{1}{2}\|\tu^n(0)\|^2_{L^2(\Omega)}-\frac{1}{2}\|\tu^n(T)\|^2_{L^2(\Omega)},\end{multline*}
we have
\begin{align*}
  A^n &=\lambda \int_0^T \langle  \tW|\tu^n|^{p-2}\tu^n,\tu^n\rangle dt+\frac{1}{2}\|\tu^n(0)\|^2_{L^2(\Omega)}-\frac{1}{2}\|\tu^n(T)\|^2_{L^2(\Omega)}+\\
  &\quad - \int^T_0\ll-\dL \tu^n(t) ,w(t) \rr dt-\int^T_0\ll-\dL w(t) ,\tu^n(t)-w(t) \rr dt \\
  & =I_1+I_2+I_3+I_4+I_5.
\end{align*}
When $n\to\infty$, taking into account~\eqref{eq:weakconv}, we observe that
\begin{itemize}
\item  $I_1$ converges (up to a subsequence, since $W\leq \omega_1$,~\eqref{eq:Lremb}, and~\eqref{eq:apriori});
\item  $I_2,I_4,I_5$ converge;
\item  in the case of $I_3$, due to weak convergence of $u^n(T)$, we have $\liminf \|\tu^n(T)\|^2_{L^2(\Omega)}\geq |\tu^n(T)|^2$.
\end{itemize}

We take upper limit in the above equation to get
\begin{multline}
\label{eq:Anest} 0\leq  \limsup_{n\to\infty}A^n\leq \lambda \int_0^T   \tW\|\tu \|^{p}_{L^p(\Omega)} dt+\frac{1}{2}\|\tu (0)\|^2_{L^2(\Omega)}-\frac{1}{2}\|\tu(T)\|^2_{L^2(\Omega)}+\\
 - \int^T_0\ll\chi ,v  \rr dt-\int^T_0\ll-\dL v(t) ,\tu (t)-v(t) \rr dt.\end{multline}
Note that $\tu_t +\chi=\lambda\tW|\tu|^{p-2}\tu$, so
\[\lambda \int_0^T   \tW\|\tu \|^{p}_{L^p(\Omega)} dt+\frac{1}{2}\|\tu (0)\|^2_{L^2(\Omega)}-\frac{1}{2}\|\tu(T)\|^2_{L^2(\Omega)}=\int_0^T \langle \chi,\tu\rangle dt.\]
This, together with~\eqref{eq:Anest}, implies
\[0\leq \int_0^T \langle \chi,\tu\rangle dt-\int_0^T \langle \chi,v\rangle dt-\int_0^T \ll-\dL v,\tu-v\rr dt= \int_0^T \ll\chi-(-\dL v),\tu-v\rr dt.\]

As we stated in Fact~\ref{facthemi},   hemicontinuity of $\dL$ implies
\[\int_0^T\ll \chi-(-\dL(\tu-\kappa w),w\rr dt\geq 0,\] where $v=\tu-\kappa w$,  $w\in  L^p( 0,T; W_{0}^{1,p}(\Omega,\omega_2))$, and $\kappa>0$ is arbitrary. Letting now $\kappa\to 0$ we get
\[ \int_0^T\ll \chi-(-\dL \tu),w\rr dt\geq 0,\]
independently of the sign of $w$. Thus,
\[\chi=-\dL \tu.\]\end{proof}


\end{document}